\documentclass[11pt,reqno]{article}
\usepackage{amsmath,amsfonts,amssymb,latexsym,epsfig}
\usepackage{a4,color,epsf}
\usepackage{graphicx,stix}
\usepackage{alltt}
\usepackage{enumerate}

\newcommand{\R}{\mathbb{R}}
\newcommand{\C}{\mathbb{C}}

\newcommand{\N}{\mathbb{N}}

\newtheorem{remark}{Remark}

\newtheorem{proposition}{Proposition}
\newtheorem{definition}{Definition}

\newenvironment{proof}{\paragraph{Proof:}}{$\Box$}

\newcommand{\e}{\ensuremath{\mathrm{e}}}

\newcommand{\psis}{{\mathcal S}}

\begin{document}
\title{Compositions of pseudo-symmetric integrators with complex coefficients for the numerical integration of differential equations}

\author{ Fernando Casas\footnote{Universitat Jaume I, IMAC \& Departament de Matem\`atiques, 12071~Castell\'on, Spain. Email: fernando.casas@uji.es}, \and 
Philippe Chartier\footnote{Univ Rennes, INRIA-MINGuS, CNRS, IRMAR, F-35000 Rennes, France. Email:Philippe.Chartier@inria.fr}, \and
Alejandro Escorihuela-Tom\`as\footnote{Universitat Jaume I, Departament de Matem\`atiques, 12071~Castell\'on, Spain. Email: alescori@uji.es}, \and
Yong Zhang\footnote{Center for Applied Mathematics, Tianjin University, China,  Email: 
sunny5zhang@163.com} 
}


%
\maketitle


\begin{abstract}
In this paper, we are concerned with the construction and analysis of a new class of methods obtained as double jump compositions with complex coefficients and projection on the real axis. It is shown in particular that the new integrators are symmetric and symplectic up to high orders if one uses a symmetric and symplectic basic method. In terms of efficiency, the aforementioned technique requires fewer stages than standard compositions of the same orders and is thus expected to lead to faster methods. \\ \\
{\bf Keywords:} Composition methods, projection on the real-axis, pseudo-symmetry, pseudo-symplecticity. 
\end{abstract}\bigskip



\section{Introduction}

Given a differential equation
\begin{equation}  \label{de.1}
   \dot{x} \equiv \frac{d x}{dt} = f(x),  \qquad x(0) = x_0,
\end{equation}   
composition methods constitute a powerful technique to raise the order of a given integrator $\psi_{\tau}$ applied to (\ref{de.1}) with time-step $\tau$, as high as might be required, by considering expressions of the form 
\begin{equation}  \label{co.1}
\phi_{\tau} = \psi_{\gamma_1 \tau} \circ \psi_{\gamma_2 \tau} \circ \cdots \circ \psi_{\gamma_s \tau},
\end{equation}
where the coefficients $\gamma_1, \gamma_2, \ldots, \gamma_s$ are appropriately chosen so as to satisfy some {\em universal algebraic conditions} \cite{hairer06gni,murua99ocf,chartier09aat}. 
It is known in particular that if $\psi_{\tau}$ is of order $k$, i.e. satisfies 
$$
\varphi_{\tau}(x_0) - \psi_{\tau}(x_0) = \mathcal{O}(\tau^{k+1}),
$$
where $\varphi_{\tau}$ denotes the
exact flow of (\ref{de.1}),  then  
$\phi_{\tau}$ will be at least of order $k+1$ (i.e., local error $k+2$) if the following two conditions are satisfied
\begin{equation}  \label{cond.1}
(i) \, \sum_{i=1}^s \gamma_i = 1 \qquad \mbox{ and } \qquad (ii)  \sum_{i=1}^s \gamma_i^{k+1} = 0.
\end{equation}
Given that these two equations have no real solution for odd $k$ and arbitrary $s$, a series of authors (e.g. \cite{suzuki90fdo,yoshida90coh}) suggested to start from a second-order method and to consider {\em symmetric compositions} only, i.e., schemes
 with coefficients satisfying the additional condition 
$$
\gamma_{s+1-i}=\gamma_i, \qquad i=1,\ldots,s.
$$
This has led to so-called triple-jump compositions ($s=3$, $\gamma_3=\gamma_1$) obtained by iterating the process described above to construct a sequence of symmetric methods with even orders (see, e.g.,  \cite{hairer06gni} pp. 44).  

In spite of its simplicity, the triple-jump rationale leads to inefficiencies for high orders as compared to methods obtained by solving directly the order conditions \cite{hairer06gni}. 
On top of this, it also suffers from the occurrence of negative time-steps, although this fact is not specific to triple-jump methods and concerns all composition or splitting methods of orders higher than two. This, of course, is a severe limiting factor for equations where the vector field (usually an operator) is not reversible, the prototypical example of which being the heat equation. To circumvent this difficulty, several authors have suggested to use complex time-steps (or complex coefficients) in the context of parabolic equations \cite{castella09smw, hansen09hos}. One indeed easily sees that, already for $s=2$, solutions of equations $(i)-(ii)$ exist  in $\mathbb{C}$. 

Generally speaking, suppose that $\psi_{\tau}$ is an integrator of order $k$, denoted $\psis_{\tau}^{[k]}$ in the sequel for clarity, and consider the
composition (\ref{co.1}) with $s=2$, 
\begin{equation}  \label{eq:S_h-C3n}
   \psis_{\tau}^{[k+1]} =
   \psis_{\gamma_1 \tau}^{[k]} \circ \psis_{\gamma_2 \tau}^{[k]}.
\end{equation}
Then, if the coefficients verify conditions $(i)-(ii)$, that is to say if 
\begin{equation}   \label{coefi.com}
\begin{array}{l}
  \gamma_1 =  \bar{\gamma}_2 \equiv \gamma= \displaystyle \frac{1}{2} + \frac{i}{2}
  \, \frac{\sin (\frac{2\ell+1}{k+1} \pi)}{1+\cos (\frac{2\ell+1}{k+1} \pi )} \quad
   \mbox{for} \quad
   \left\{ \begin{array} {ll}
   -\frac{k}{2} \leq \ell \leq \frac{k}{2}-1 & \mbox{if} \ k \ \mbox{is even}\\
   -\frac{k+1}{2} \leq \ell \leq \frac{k-1}{2} & \mbox{if} \ k \ \mbox{is odd}
    \end{array} \right.
\end{array},
\end{equation}
then (\ref{eq:S_h-C3n}) results in a method of order $k+1$, which can subsequently be used
to generate recursively higher order composition schemes by applying the same procedure.  The choice $\ell=0$, 
\begin{align} \label{eq:gam}
  \gamma = \gamma^{[k]}:= \frac{1}{2} + \frac{i}{2} \frac{\sin \left(\frac{\pi}{k+1}\right)}{1+ \cos\left(\frac{\pi}{k+1}\right)} = 
  \frac{1}{2} + \frac{i}{2} \tan \left( \frac{\pi}{2(k+1)} \right) = \frac{1}{2 \cos\left(\frac{\pi}{2(k+1)}\right)} e^{\frac{\pi}{2(k+1)}}
\end{align}
gives the solutions with
the smallest phase. If we start with a symmetric method of order 2, $\psis_{\tau}^{[2]}$, and apply composition (\ref{eq:S_h-C3n}) with
corresponding coefficients (\ref{eq:gam}), 
we can construct the following sequence of methods:
\[
   \psis_{\tau}^{[2]} \longrightarrow \psis_{\tau}^{[3]} \longrightarrow \psis_{\tau}^{[4]} \longrightarrow \psis_{\tau}^{[5]} \longrightarrow \psis_{\tau}^{[6]}, 
\]
all of which have coefficients with positive real part  \cite{hansen09hos}. 
The final method of order $6$ involves $16$ evaluations of the basic scheme $\psis_{\tau}^{[2]}$. By contrast, there are
composition methods of order $6$ (both with real and complex coefficients) involving just $7$ evaluations of $\psis_{\tau}^{[2]}$ \cite{blanes13oho,yoshida90coh}. It is thus apparent that
this direct approach does not lead to cost-efficient high-order schemes.

One should remark that the composition (\ref{eq:S_h-C3n}) does not provide a time-symmetric method, i.e., 
$\psis_{-\tau}^{[k+1]} \circ \psis_{\tau}^{[k+1]}$ is \emph{not} the identity map, even if  $\psis_{\tau}^{[k]}$ happens to be symmetric. As 
we have mentioned before,
symmetry allows to raise the order by two at each iteration by considering the triple-jump composition
\begin{equation}\label{suzu_n}
   \psis_{\tau}^{[2k+2]} =
   \psis_{\gamma_1 \tau}^{[2k]} \circ \psis_{\gamma_2 \tau}^{[2k]} \circ  \psis_{\gamma_1
   \tau}^{[2k]}
\end{equation}
starting from a symmetric method. 
Apart from the real solution, the complex one with the smallest
phase is
\begin{equation}   \label{tj2p}
 \gamma_1 =\frac{\e^{i\pi/(k+1)}}{2^{1/(k+1)}-2 \, \e^{i\pi/(k+1)}}, \qquad
 \gamma_2 = 1-2 \gamma_1,
\end{equation}
and symmetric methods up to order $8$ with coefficients having positive
real part are possible if one starts with a symmetric second-order scheme\footnote{It is actually possible to reach order 14 if, in the construction, one uses formula \eqref{suzu_n} alternatively with coefficients $\gamma_1, \gamma_2$ and coefficients $\bar \gamma_1, \bar \gamma_2$ \cite{castella09smw}.}. These order barriers has been rigorously proved in \cite{blanes13oho}.

The simple third-order scheme (\ref{eq:S_h-C3n}) corresponding to $k=2$ has been in fact rediscovered several times in the literature
\cite{bandrauk91ies,castella09smw,chambers03siw,hansen09hos,suzuki92gto}. 
In particular, it was shown in \cite{chambers03siw} that the method, when applied to the two-body Kepler problem, behaves indeed as a fourth-order integrator, the reason being 
attributed to the fact that the main error term in the asymptotic expansion is purely imaginary. In this note we elaborate further the analysis and
provide a comprehensive study of the general composition (\ref{eq:S_h-C3n}), paying special attention to the qualitative properties the
method shares with the continuous system (\ref{de.1}).
In addition, we show how it is possible
to combine compositions and a trivial linear combination to raise the order, while still preserving the qualitative properties of the basic integrator up to an order higher
than of the method itself.

\section{Composition and pseudo-symmetry or pseudo-symplecticity}

In what follows, we will assume for convenience that all values of $x$ in (\ref{de.1}) lie in a compact set $K$ where the function $f$ is {\em smooth}. 
Before starting the analysis, it is worth recalling the notions of adjoint method and symplectic flow.

The \emph{adjoint method} $\psi_{\tau}^*$ of a given method is the inverse map of the original integrator with reversed
time step $-\tau$, i.e., $\psi_{\tau}^* := \psi_{-\tau}^{-1}$. A symmetric method satisfies $\psi_{\tau}^* = \psi_{\tau}$ \cite{chartier15sm,hairer06gni}. 

The vector field $f$ in (\ref{de.1}) is Hamiltonian if there exists a function $H(x)$ such that $f = J \nabla_x H(x)$, where $x = (q,p)^T$ and
$J$ is the basic canonical matrix. Then, the exact flow of (\ref{de.1}) is a symplectic transformation, $\varphi_t'(x)^T J \varphi_t'(x) = J$ for
$t \ge 0$ \cite{blanes16aci,sanz-serna94nhp}.

It then makes sense introducing the following definitions, taken from \cite{chartier98rbs} and \cite{aubry98psr}:
\begin{definition}  \label{def1}
Let $\psi_\tau$ be a smooth and consistent integrator:
\begin{enumerate}
\item it is  {\em pseudo-symmetric} of {\em pseudo-symmetry order} $q$ if for all sufficiently small $\tau$, the following relation holds true:
\begin{align}
\label{eq:pseudosym0}
\psi_\tau^* = \psi_{\tau} + {\cal O}(\tau^{q+1}),
\end{align}
where the constant in the ${\cal O}$-term depends on bounds of derivatives of $\psi$ on $K$. 
\item it is  {\em pseudo-symplectic} of {\em pseudo-symplecticity order} $r$ if for all sufficiently small $\tau$, the following relation holds true
when $\psi_{\tau}$ is applied to a Hamiltonian system:
\begin{align}
\label{eq:pseudosymplec}
(\psi_\tau')^T \, J \, \psi_{\tau}' = J + {\cal O}(\tau^{r+1}),
\end{align}
where the constant in the ${\cal O}$-term depends on bounds of derivatives of $\psi$ on $K$. 
\end{enumerate}
\end{definition}
\begin{remark}
A symmetric method is pseudo-symmetric of any order $q \in \N$, whereas a method of order $k$ is pseudo-symmetric of order $q \geq k$. A similar statement holds for symplectic methods. 
\end{remark}

As a first illustration of Definition \ref{def1}, let us consider again a symmetric 2nd-order method $\psis_{\tau}^{[2]}$ and form the composition
\[
\psi_{\tau}^{[3]} = \psis_{\gamma \tau}^{[2]} \circ \psis_{\bar{\gamma} \tau}^{[2]}
\]  
with $\gamma = \frac12 + i \frac{\sqrt{3}}{6}$. Then, if the vector field $f$ under consideration is real-valued, its real part
\[
 \Re(\psi_{\tau}^{[3]})  = \frac12 \left( \psi_{\tau}^{[3]} + \overline{\psi}_{\tau}^{[3]} \right) = \frac{1}{2}
 \left( \psis_{\gamma \tau}^{[2]} \circ \psis_{\bar \gamma \tau}^{[2]} + \psis_{\bar \gamma \tau}^{[2]} \circ \psis_{ \gamma \tau}^{[2]} \right).
\]
is a method of order 4 and pseudo-symmetric of pseudo-symmetry order 7. This result is a consequence of the fact that
\[
     (\psi_{\tau}^{[3]})^*  = \overline{\psi}_{\tau}^{[3]}
\]     
and the following general
statement, which lies at the core of the construction procedure described in this paper.

\begin{proposition} \label{prop:main}
Let $\psi_{\tau}$ be any consistent smooth method for equation (\ref{de.1}) and consider the new method
$$
R_{\tau} = \frac12 \left( \psi_{\tau} + \psi_{\tau}^*\right).
$$
Assume also that $\psi_{\tau}$ is pseudo-symmetric of order $q$. Then $R_{\tau}$ is  of pseudo-symmetry order $2q+1$. 
If $\psi_\tau$ is furthermore of pseudo-symplecticity order $r$, then $R_\tau$ is 
of pseudo-symplecticity order $\min(2 q+1,r)$.
\end{proposition}
\begin{proof}
By assumption, there exists a smooth function $(\tau,x)  \mapsto \delta_\tau(x)$, defined for all $x$ in a compact set $K$ and for all sufficiently small real $\tau$, such that 
\begin{align}
\label{eq:pseudosym}
\psi_\tau^* = \psi_{\tau} + \tau^{q+1} \delta_\tau \quad \mbox{ or } \quad \psi_{-\tau}^{-1} = \psi_{\tau} + \tau^{q+1} \delta_{\tau}    \quad \mbox{ or } \quad \psi_\tau^{-1} = \psi_{-\tau} + (-\tau)^{q+1} \delta_{-\tau}, 
\end{align}
so that 
$$
R_\tau = \psi_\tau + \frac12 \tau^{q+1} \delta_\tau.
$$
Composing the third  relation of \eqref{eq:pseudosym} from the left by $\psi_{\tau}$, we obtain 
\begin{align} 
{\rm id} = \psi_{\tau} \circ \psi_{-\tau} + (-\tau)^{q+1} \psi_{\tau}' \circ \psi_{-\tau} \cdot \delta_{-\tau} +  {\cal O}(\tau^{2(q+1)}),
\end{align}
where the ${\cal O}$-term depends on bounds of the derivatives of $\psi_\tau$ and $\delta_\tau$  on $K$. Similarly, composing the second  relation of \eqref{eq:pseudosym} from the right by $\psi_{-\tau}$, we get 
\begin{align} 
{\rm id} = \psi_{\tau} \circ \psi_{-\tau} + \tau^{q+1}   \delta_{\tau} \circ \psi_{-\tau}.
\end{align}
As a consequence, we have 
$$
 \tau^{q+1}   \delta_{\tau} \circ \psi_{-\tau}  = (-\tau)^{q+1} \psi_{\tau}' \circ \psi_{-\tau} \cdot \delta_{-\tau} +  {\cal O}(\tau^{2(q+1)}). 
$$
We are then in position to write 
\begin{eqnarray*}
R_{\tau} \circ \,R_{-\tau} &=& \left( \psi_{\tau} +  \frac 12 \tau^{q+1} \delta_{\tau} \right) \circ   \left( \psi_{-\tau} +  \frac 12 (-\tau)^{q+1} \delta_{-\tau} \right) \\
&=& \psi_{\tau} \circ \, \psi_{-\tau} + \frac12 (-\tau)^{q+1} \psi'_{\tau} \circ \, \psi_{-\tau} \cdot  \delta_{-\tau}  
+ \frac 12 \tau^{q+1} \delta_{\tau} \circ \, \psi_{-\tau}  + {\cal O}(\tau^{2(q+1)})  \\
&=& {\rm id} +  {\cal O}(\tau^{2(q+1)}),
\end{eqnarray*}
which proves the first statement. Now, if $\psi_\tau$ is in addition of pseudo-symplecticity order $r$, then its adjoint $\psi^*_\tau$ is also of pseudo-symplecticity order $r$, so that relation (\ref{eq:pseudosym}) leads to 
\begin{align*}
J + {\cal O}(\tau^{r+1})&= (\partial_x \psi^*_\tau)^T J \partial_x \psi^*_\tau = (\psi_\tau' + \tau^{q+1} \delta'_\tau)^T J (\psi_\tau' + \tau^{q+1} \delta'_\tau) \\
&=J + {\cal O}(\tau^{r+1}) +\tau^{q+1} \left(  (\delta'_\tau )^T \, J \, \psi_\tau'  + (\psi_\tau' )^T \, J \, \delta'_\tau \right) + {\cal O}(\tau^{2(q+1)}),
\end{align*}
which implies that 
$$
\tau^{q+1} \left(  (\delta'_\tau )^T \, J \, \psi_\tau'  + (\psi_\tau' )^T \, J \, \delta'_\tau \right) =  {\cal O}(\tau^{\min(2(q+1),r+1)}).
$$
As an immediate consequence, we have that 
\begin{align*}
(R_\tau')^T J R_\tau' 
=J +{\cal O}(\tau^{\min(2(q+1),r+1)})
\end{align*}
which proves the second statement.
\end{proof}
\\ \\
This result can be rendered more specific as follows: \nopagebreak
\begin{proposition} \label{prop.2}
Let $\psis_{ \tau}^{[2n]}$ be a smooth method of order $2n \geq 2$ and pseudo-symmetry order $q \geq 2n+1$. Let us consider 
the composition method
\begin{equation}   \label{compo.2}
   \psi_{\tau}^{[2n+1]} = \psis_{\gamma_1 \tau}^{[2n]} \circ \, \psis_{\gamma_2 \tau}^{[2n]},
\end{equation}
where the coefficients $\gamma_1$ and $\gamma_2$ satisfy both relations $\gamma_1 + \gamma_2 = 1$ and $\gamma_1^{2n+1} + \gamma_2^{2n+1} = 0$. Then the method
\begin{equation}  \label{compo.3}
   \hat R_{\tau} =\frac{1}{2} \left( \psi_{\tau}^{[2n+1]} + \overline{\psi}_{\tau}^{[2n+1]} \right) 
\end{equation}
is of order 
\begin{align}
\label{eq:order}
\left\{
\begin{array}{ll}
2n+1 & \mbox{ if } \; q  = 2n+1,\\
2n+2 & \mbox{ if } \; q  \geq 2n+2
\end{array}
\right.
\end{align}
 when the vector field $f$ in  (\ref{de.1}) is real, and of pseudo-symmetry order 
 \begin{align}
 \label{eq:symetry}
\left\{
\begin{array}{ll}
2n+1 & \mbox{ if } \; q  = 2n+1,\\
\min(q,4n+3) & \mbox{ if } \; q \geq 2n+2.
\end{array}
\right.
\end{align}
If in addition, $f$ is
a (real) Hamiltonian vector field and $\psis_{\tau}^{[2n]}$ is of pseudo-symplecticity order $r$, then $\hat{R}_{\tau}$ is of pseudo-symplecticity order  
 \begin{align}
 \label{eq:symplecticity}
\left\{
\begin{array}{ll}
\min(r,2n+1) & \mbox{ if } \; q  = 2n+1,\\
\min(q,r,4n+3) & \mbox{ if } \; q \geq 2n+2.
\end{array}
\right.
\end{align}
\end{proposition}
\begin{remark}
Note that in Proposition \ref{prop.2}, one has necessarily $q \geq 2n+1$. This can be seen straightforwardly by a direct computation of 
${\cal S}_{-\tau}^{[2n]} \circ {\cal S}_\tau^{[2n]}(x)$  with 
${\cal S}_\tau^{[2n]}(x) = \varphi_\tau (x)+ \tau^{2n+1} C(x) + {\cal O}(\tau^{2n+2})$.
\end{remark}
\begin{proof}
Noticing that $\gamma_1$ and $\gamma_2$ are complex conjugate and \eqref{de.1} is real,  and taking into account that $\psis_{ \tau}^{[2n]}$ is
of pseudo-symmetry order $q$, we have
\begin{align}
\label{eq:conjadjointq}
\overline{\psi}_{\tau}^{[2n+1]} &= \psis_{\gamma_2 \tau}^{[2n]} \, \circ \, \psis_{\gamma_1 \tau}^{[2n]} = \left((\psis_{\gamma_2 \tau}^{[2n]})^* +{\cal O}(\tau^{q+1})\right)\, \circ \,  \left((\psis_{\gamma_1 \tau}^{[2n]})^*+{\cal O}(\tau^{q+1})\right) \nonumber \\
&= (\psis_{\gamma_2 \tau}^{[2n]})^* \, \circ \,  (\psis_{\gamma_1 \tau}^{[2n]})^* +{\cal O}(\tau^{q+1})= (\psi_{\tau}^{[2n+1]})^* +  {\cal O}(\tau^{q+1}).
\end{align}
Moreover, by construction, $\psi_{\tau}^{[2n+1]}$ is at least of order $2n+1$, so that 
\begin{align}
\label{eq:conjadjointn}
\psi_{\tau}^{[2n+1]} +  {\cal O}(\tau^{2n+2}) = \overline{\psi}_{\tau}^{[2n+1]} = (\psi_{\tau}^{[2n+1]})^* +  {\cal O}(\tau^{2n+2}),
\end{align}
and altogether
\begin{align}
\label{eq:conjadjoint}
\hat R_{\tau}  &= R_{\tau}  +  {\cal O}(\tau^{\max(2n+2,q+1)}).
\end{align}
Now, since the pseudo-symmetry order of $\psi_{\tau}^{[2n+1]}$ is at least $2n+1$, the method 
$$
R_{\tau} = \frac{1}{2} \left( \psi_{\tau}^{[2n+1]} + (\psi_{\tau}^{[2n+1]})^* \right) 
$$
is, according to Proposition \ref{prop:main},  of pseudo-symmetry order $4n+3$ and of pseudo-symplecticity order $\min(4n+3,q)$. The first (\ref{eq:order}), second (\ref{eq:symetry}) and third (\ref{eq:symplecticity}) statements on orders then follow from   (\ref{eq:conjadjoint}).
\end{proof}

\

In the Appendix we provide an alternative proof of Proposition \ref{prop.2} based on the Lie formalism, which allow us,
in addition, to generalise the previous result on pseudo-symplecticity to other geometric properties the continuous system may possess 
(such as in volume preserving flows, isospectral flows, differential equations evolving on Lie groups, etc.).

Notice that, according with Proposition \ref{prop.2}, if we start from  $n=1$, that is to say from a basic symmetric ($q=+\infty$) and symplectic ($r=+\infty$) method of order $2$, we get a method of order $4$ that is pseudo-symmetric and
pseudo-symplectic of order $7$ just by considering the simple composition (\ref{compo.2}) and taking the real part of the output at each time step. 
If this technique is applied to a symmetric and symplectic method of order $4$, i.e. with $n=2$, then $\hat R_{\tau}$ 
is of order $6$ and pseudo-symmetric and
pseudo-symplectic of order $11$. 

Let us consider, in particular, the $4^{th}$-order symmetric scheme (\ref{suzu_n}) with $k=2$ as basic scheme. Then, the resulting $6^{th}$-order integrator $\hat R_{\tau}$ only
requires the evaluation of $6$ second-order methods $\psis_{\tau}^{[2]}$, whereas the corresponding $6^{th}$-order scheme obtained by the triple-jump technique 
involves $9$ evaluations. This number is reduced to $7$ by considering general compositions of $\psis_{\tau}^{[2]}$ \cite{blanes13oho}. If we take this $6^{th}$-order
composition of $7$ schemes as the basic method $\psis_{\tau}^{[6]}$, the resulting integrator of order $8$, $\hat R_{\tau}$, involves the evaluation of $14$ $\psis_{\tau}^{[2]}$, 
whereas $15$ evaluations are required by pure composition methods. Notice that  $\hat R_{\tau}$ is
pseudo-symmetric and pseudo-symplectic of order 15, so that for values of $\tau$ sufficiently small, it preserves effectively the symmetry up to round-off error while the drift in energy for Hamiltonian systems is hardly noticeable.

\section{Families of pseudo-symplectic methods}

There is another possibility to increase the order, though, and it consists in applying the technique of Proposition \ref{prop.2} recursively. 
Thus, if denote by $\hat{R}_{\tau}^{(1)} \equiv \hat{R}_{\tau}$ the method of eq. (\ref{compo.3}), we propose to apply the following recurrence:
\begin{equation}   \label{foit}
\begin{aligned}
   & \mbox{For } \, i = 2, 3, \ldots \\
   &  \qquad  \qquad \Phi_{\tau} ^{(i)}= \hat R_{\gamma^{[2i]} \tau}^{(i-1)} \circ \hat R_{\bar{\gamma}^{[2i]} \tau}^{(i-1)}  \\
   &  \qquad   \qquad \hat{R}_{\tau}^{(i)}=    \frac{1}{2} \left( \Phi_{\tau}^{(i)} + \overline{\Phi}_{\tau}^{(i)} \right) 
\end{aligned}
\end{equation}
where $\gamma^{[2i]}$ is given by (\ref{eq:gam}).
 Then, according with Proposition \ref{prop.2}, it is possible to raise the order up to the pseudo-symmetry order of the 
 underlying basic method $\psis_{\tau}^{[2n]}$. Thus, in particular, the maximum order one can achieve by applying this technique to 
the basic symmetric method $ \psis_{\tau}^{[2]}$ is 7, whereas if we start with a basic symmetric method of order 4, $ \psis_{\tau}^{[4]}$, the maximum
order is $11$. It is $15$ from a symmetric method $ \psis_{\tau}^{[6]}$ of order $6$ and so on and so forth.  

To give an assessment of the computational cost of the methods obtained by applying this type of composition, we notice that
the computation of $\Phi_{\tau}^{(i)}$ and $\overline{\Phi}_{\tau}^{(i)}$ required to form $\hat{R}_{\tau}^{(i)}$ by (\ref{foit}) at the intermediate stages
can be done in parallel, whereas at the final stage it only requires taking the real part. Thus, the method of order 6 constructed recursively from $ \psis_{\tau}^{[2]}$ only requires
the effective computation of 4 basic methods $\psis_{\tau}^{[2]}$. 

Starting from a symmetric second-order method $\psis_{\tau}^{[2]}$, say Strang splitting for instance, it is important to monitor the sign of the real part of all coefficients involved in the previous iteration. It is immediate to see that in the recursive construction   
$$
\psis_{\tau}^{[2]} \rightarrow \hat{R}_{\tau}^{(1)}  \rightarrow \hat{R}_{\tau}^{(2)}  \rightarrow \hat{R}_{\tau}^{(3)}
$$
envisaged in the recurrence (\ref{foit}), the basic method $\psis_{\tau}^{[2]}$ is used with the following coefficients
\begin{align*}
i=1:  & \; \gamma^{[2]} , \quad \bar{\gamma}^{[2]}  \\
i=2:  & \; \gamma^{[4]} \gamma^{[2]} , \quad \bar{\gamma}^{[4]} \gamma^{[2]} , \quad \gamma^{[4]} \bar{\gamma}^{[2]} , \quad \bar{\gamma}^{[4]} \bar{\gamma}^{[2]}  \\
i=3:  & \; \gamma^{[6]}  \gamma^{[4]} \gamma^{[2]} , \; \gamma^{[6]} \bar{\gamma}^{[4]} \gamma^{[2]} , \; \gamma^{[6]} \gamma^{[4]} \bar{\gamma}^{[2]} , \; \gamma^{[6]} \bar{\gamma}^{[4]} \bar{\gamma}^{[2]} , \; \bar{\gamma}^{[6]}  \gamma^{[4]} \gamma^{[2]} , \; \bar{\gamma}^{[6]} \bar{\gamma}^{[4]} \gamma^{[2]} , \; \bar{\gamma}^{[6]} \gamma^{[4]} \bar{\gamma}^{[2]} , \; \bar{\gamma}^{[6]} \bar{\gamma}^{[4]} \bar{\gamma}^{[2]} 
\end{align*} 
Given the expression of $\gamma^{[k]}$ (see (\ref{eq:gam})), these coefficients have arguments of the form 
$$
\frac{\pi}{2} \sum_{j=1}^i \pm \frac{1}{2j+1} = \frac{\pi}{2} \Big (\pm \frac13 \pm  \frac15 \pm  \cdots \pm \frac{1}{2i+1} \Big), \qquad i=1,2,3,
$$
so that their maximum argument is 
$$
\frac{\pi}{2} \sum_{j=1}^3 \frac{1}{2j+1}.
$$
For all the coefficients to have positive real parts, a necessary and sufficient condition is thus that 
$$
\sum_{j=1}^3 \frac{1}{2j+1} \leq 1.
$$
It clearly holds for methods $\hat{R}_{\tau}^{(1)} $, $\hat{R}_{\tau}^{(2)} $ and $\hat{R}_{\tau}^{(3)}$, of respective orders $4$, $6$ and $7$, since 
$1/3+1/5+1/7=71/105$. Similarly, starting form a symmetric method of order $4$ having real or complex coefficients with maximum argument $\theta_4$, the condition becomes 
$$
\frac{2 \theta_4}{\pi} + \sum_{j=1}^4 \frac{1}{2j+3} = \frac{2 \theta_4}{\pi}  + \frac{1888}{3465} \leq 1.
$$
For instance, suppose that $f$ in (\ref{de.1}) can be split as $f(x) = f_a(x) + f_b(x)$, so that the exact $\tau$-flows $\varphi_{\tau}^{[a]}$ and
$\varphi_{\tau}^{[b]}$ corresponding to $f_a$ and $f_b$, respectively, can be computed exactly. Then, the following composition
\begin{equation}  \label{s4sim}
  \psis_{\tau}^{[4]} = \varphi_{b_1 \tau}^{[b]} \circ \,  \varphi_{a_1 \tau}^{[a]} \circ \, \varphi_{b_2 \tau}^{[b]} \circ \,  \varphi_{a_2 \tau}^{[a]} \circ \, 
  \varphi_{b_3 \tau}^{[b]} \circ \,  \varphi_{a_2 \tau}^{[a]} \circ \, \varphi_{b_2 \tau}^{[b]} \circ \,  \varphi_{a_1 \tau}^{[a]} \circ \, \varphi_{b_1 \tau}^{[b]} 
\end{equation}
with 
\[
   b_1=\frac{1}{10}-\frac{1}{30} i,  \quad b_2=\frac{4}{15}+\frac{2}{15} \, i, \quad b_3={\frac{4}{15}}-\frac15 i \quad
   \mbox{  and } \quad  a_1=a_2=a_3=a_4 = \frac{1}{4}
\]
provides a 4th-order symmetric scheme (see \cite{castella09smw}). Taking  (\ref{s4sim}) as basic method 
we get $\max_{i=1,2,3} \mathrm{Arg}(b_i) = \arccos \left( 4/5 \right) $ so that 
$$
\frac{2 \theta_4}{\pi}  + \frac{1888}{3465} < 0.409666 + \frac{1888}{3465}< 0.96 <  1
$$
and thus all methods 
$$
\psis_{\tau}^{[4]} \rightarrow \hat{R}_{\tau}^{(1)}  \rightarrow \hat{R}_{\tau}^{(2)}  \rightarrow \hat{R}_{\tau}^{(3)} \rightarrow \hat{R}_{\tau}^{(4)}
$$
of respective orders $4$, $6$, $8$, $10$ and $11$ obtained by the procedure  (\ref{foit}) have  all their coefficients with positive real parts. As far as the $f_a$ part is concerned, the maximum argument is less than $0.55 \frac{\pi}{2}$.

\section{Numerical experiments}
%

In this section we illustrate the previous results on several numerical examples, comprising Hamiltonian systems and partial differential equations of
evolution previously discretised in space. 

\subsection{Harmonic oscillator}

We first consider the simple
harmonic oscillator, with Hamiltonian 
\[
H=T(p)+V(q)=\frac{1}{2}p^{2}+\frac{1}{2}q^{2}. 
\]
If we denote by $M_{X}(\tau)$ the exact matrix evolution associated with the Hamiltonians $X=H$, $T$
and $V$, i.e., $(q(\tau),p(\tau))^{T}=M_{X}(\tau)(q(0),p(0))^{T}$, then
\[
M_{H}(\tau)=\left( 
\begin{array}{rr}
\cos (\tau) & \sin (\tau) \\ 
-\sin (\tau) & \cos (\tau)
\end{array}
\right) ,\quad M_{T}(\tau)=\left( 
\begin{array}{ll}
1 & \tau \\ 
0 & 1
\end{array}
\right) ,\quad M_{V}(\tau)=\left( 
\begin{array}{rr}
1 & 0 \\ 
-\tau & 1
\end{array}
\right), 
\]
respectively. We take as basic symmetric (and symplectic) scheme the leapfrog/Strang splitting:
\begin{equation}   \label{strang.ho}
\psis_{\tau}^{[2]} = M_{T}(\tau/2)M_{V}(\tau)M_{T}(\tau/2)
\end{equation}
and compute the first three iterations in (\ref{foit}). In Table \ref{table.1} we collect the main term
in the truncation error for the resulting integrators $\hat{R}_{\tau}^{(i)}$, $i=1,2,3$. 
We also check their time-symmetry and the preservation of the symplectic character of the approximate solution matrix
by computing its determinant (a $2 \times 2$ matrix $A$ is symplectic iff $\det(A) = 1$). One can observe that these results are in agreement with the previous estimates.

\begin{table}[t]  
\centering
{\footnotesize 
\begin{tabular}{|c|ccc|}   \hline
$\psi_{\tau}$ & $M_{H}(\tau)- \psi_{\tau}$ & $\psi_{\tau} \circ \psi_{-\tau} - I_2$ & $\det
\left( \psi_{\tau} \right)$ \\ \hline\hline
$\hat{R}_{\tau}^{(1)}$ & $\left( 
\begin{array}{cc}
0 & -\frac{1}{180} \\ 
-\frac{1}{120} & 0
\end{array}
\right) \tau^{5}$  & $\left( 
\begin{array}{cc}
-\frac{1}{1728} & 0 \\ 
0 & -\frac{1}{1728}
\end{array}
\right) \tau^{8}$ & $1 - \frac{1}{1728}\tau^{8}$ \\ 
$\hat{R}_{\tau}^{(2)}$ & $\left( 
\begin{array}{cc}
0 & 3.8\,\times 10^{-5} \\ 
5.1\times 10^{-5} & 0
\end{array}
\right) \tau^{7}$ 
 & $\left( 
\begin{array}{cc}
5.4\, \times 10^{-6}  & 0 \\ 
 0 & 5.4\times 10^{-6}
\end{array}
\right) \tau^{8}$ & $1+ 5.4\times 10^{-6} \, \tau^{8}$ \\ 
$\hat{R}_{\tau}^{(3)}$ & $\left( 
\begin{array}{cc}
5.8\,\times 10^{-9} & 0 \\ 
0 & 5.8\,\times 10^{-9}
\end{array}
\right) \tau^{8}$  & 
$\left( 
\begin{array}{cc}
-1.\,1\times 10^{-8} & 0 \\ 
0 & -1.\,1\times 10^{-8}
\end{array}
\right) \tau^{8}$ & $1-1.\,1\times 10^{-8} \, \tau^{8}$ \\ \hline
\end{tabular}
}
\caption{\small
Main term in the truncation error, degree of  symmetry and symplecticity for schemes $\hat{R}_{\tau}^{(i)}$
obtained from the basic leapfrog integrator for the simple harmonic oscillator. $I_2$ stands for the $2 \times 2$ identity matrix.
\label{table.1}}
\end{table}

Next we take initial conditions $(q,p) = (2.5,0)$, integrate until the final time $t_f = 10^4$ with $\psis_{\tau}^{[2]}$, $\hat{R}_{\tau}^{(1)}$,  and
$\hat{R}_{\tau}^{(2)}$ and compute the relative error in energy along the evolution. The result is depicted in Figure \ref{figu.ho}. We see
that for $\hat{R}_{\tau}^{(1)}$ and $\hat{R}_{\tau}^{(2)}$ the error in energy is almost constant 
for a certain
period of time, and then there is a secular growth proportional to $\mathcal{O}(\tau^7)$.

\begin{figure}[h!] 
\begin{center}
\includegraphics[width=15.5cm, height=10.7cm]{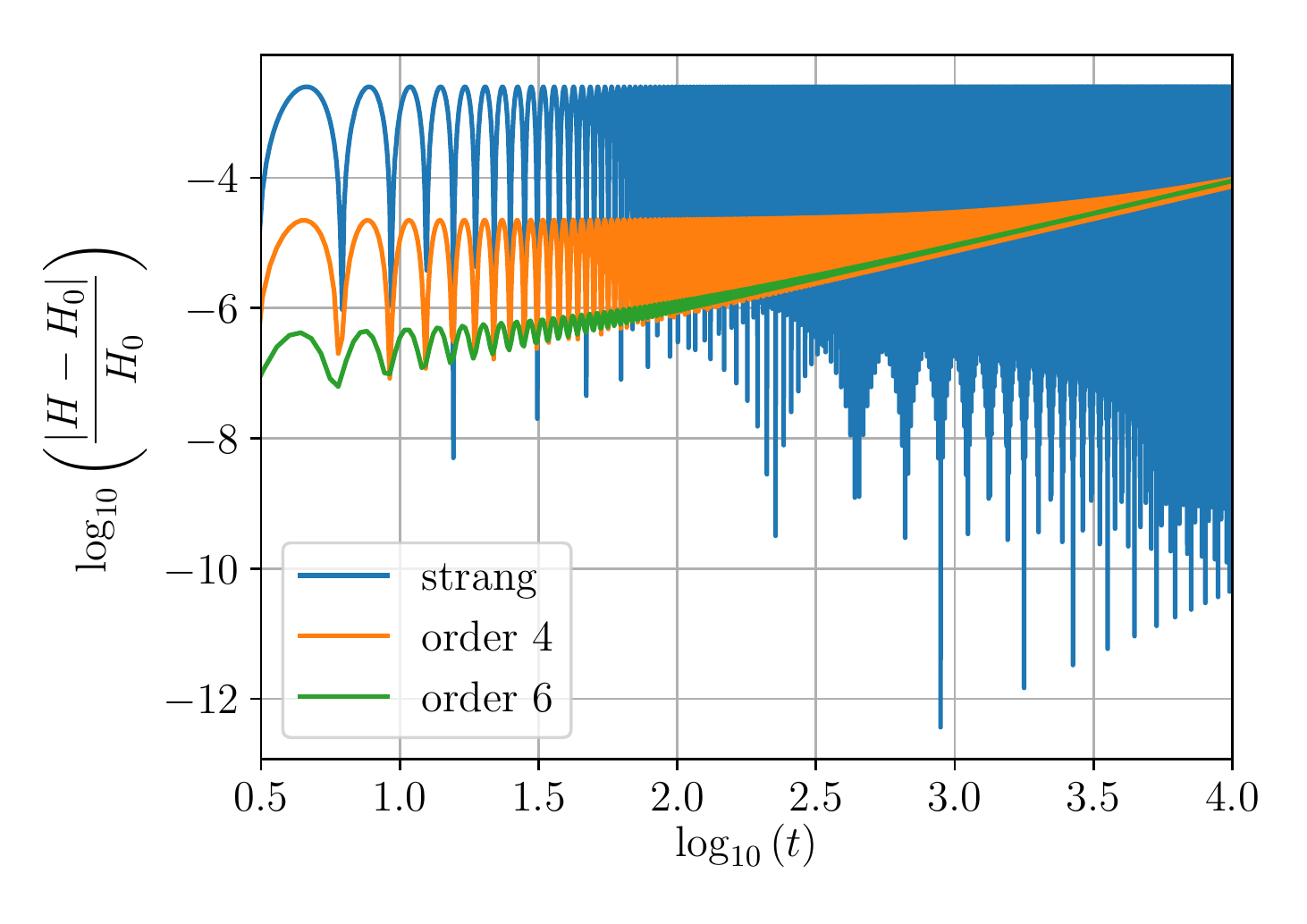}
 
\end{center}
\caption{\small Error in energy along the integration for the harmonic oscillator taking (\ref{strang.ho}) as the basic integrator in the
sequence (\ref{foit}). \label{figu.ho}}
\end{figure}

\subsection{Kepler problem}
Next, we consider  the two-dimensional
Kepler problem with Hamiltonian
\begin{equation}   \label{eq.HamKepler}
   H(q,p) = T(p) + V(q) = 
	\frac{1}{2} p^T p - \mu \frac{1}{r}.
\end{equation}
Here $q=(q_1,q_2), p=(p_1,p_2)$, $\mu=GM$, $G$ is the gravitational constant and $M$ is the sum of the masses of the two bodies. Taking $\mu=1$ and initial conditions
\begin{equation}\label{eq.1.12}
  q_1(0) = 1- e, \quad q_2(0) = 0, \quad p_1(0) = 0, \quad p_2(0) = \sqrt{\frac{1+e}{1-e}},
\end{equation}
if $0 \le e < 1$, then the solution is periodic with period $2 \pi$, and the trajectory is an ellipse of eccentricity $e$. Note that the gradient function must here be implemented carefully so as to be analytic for complex values of $z=q_1^2+q_2^2$ . Here, we define it using the following determination of the complex logarithm (analytic on the complex plane outside the negative real axis):
$$
\forall \, (x,y) \in \R^2 \mbox{ s.t. } x+i y \notin \R_{-}, \qquad L(x+iy) = \log|x+iy| + 2 i \arctan\left(\frac{y}{x+|x+i y|}\right).
$$
As a consequence, the analytic continuation of the function $1/r^3=1/(q_1^2+q_2^2)^{3/2}$ writes
$$
\exp\Big(-\frac32 L(x+iy)\Big),
$$
where $x=\Re(q_1^2+q_2^2)$ and $y=\Im(q_1^2+q_2^2)$. 

Here, as with the harmonic oscillator, we take as basic method the 2nd-order Strang splitting
\begin{equation}  \label{strang1}
    \mathcal{S}_{\tau}^{[2]} = \varphi_{\tau/2}^{[a]} \circ \, \varphi_{\tau}^{[b]} \circ \, \varphi_{\tau/2}^{[a]},
\end{equation}
where $\varphi_{\tau}^{[a]}$ (respectively, $\varphi_{\tau}^{[b]}$) corresponds to the exact solution obtained by integrating the kinetic 
energy $T(p)$ (resp., potential energy $V(q)$) in  (\ref{eq.HamKepler}).
   
We take $e=0.6$, integrate until the final time $t= 20$ with Strang and the schemes obtained by the recursion
(\ref{foit}) with $i=1,2,3$ for several time steps and compute the relative error in energy at the final time. Figure \ref{fig:kepler}
show this error as a function of the inverse of the step size $1/\tau$ to illustrate the order of convergence: order 2 for Strang, order 4 for
$\hat{R}_{\tau}^{(1)}$ and order 6 for $\hat{R}_{\tau}^{(2)}$. For $\hat{R}_{\tau}^{(3)}$, and contrary to what happens to the harmonic oscillator, the observed numerical order is higher than expected, varying between 7 and 8. We do not have at present a theoretical explanation for this
phenomenon. Figure \ref{fig:kepler} (right) depicts the time evolution of this error when the final time is $t=10^4$.

\begin{figure}[h!] 
\begin{center}
\includegraphics[width=7.5cm, height=5.7cm]{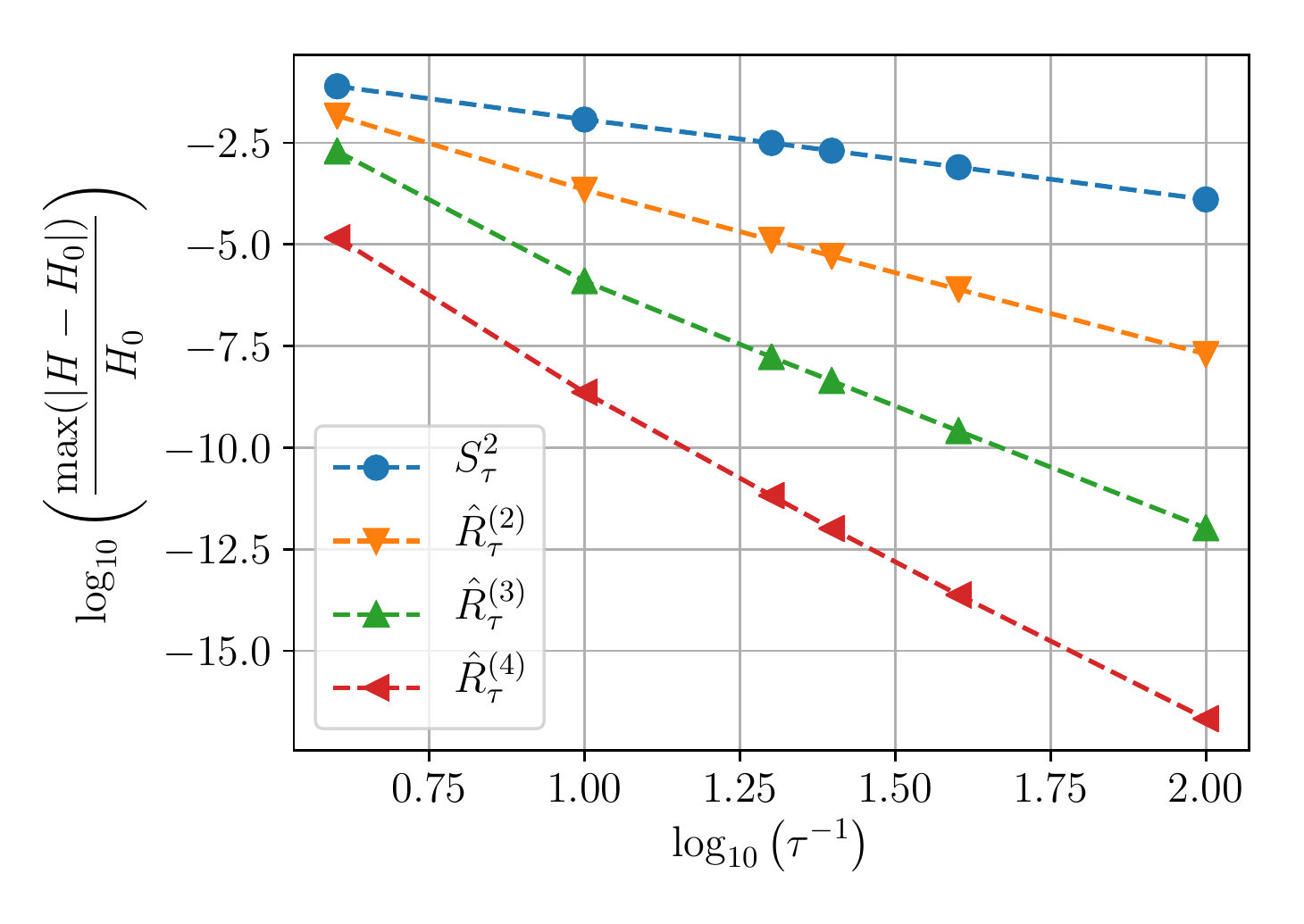}
\includegraphics[width=7.5cm, height=5.7cm]{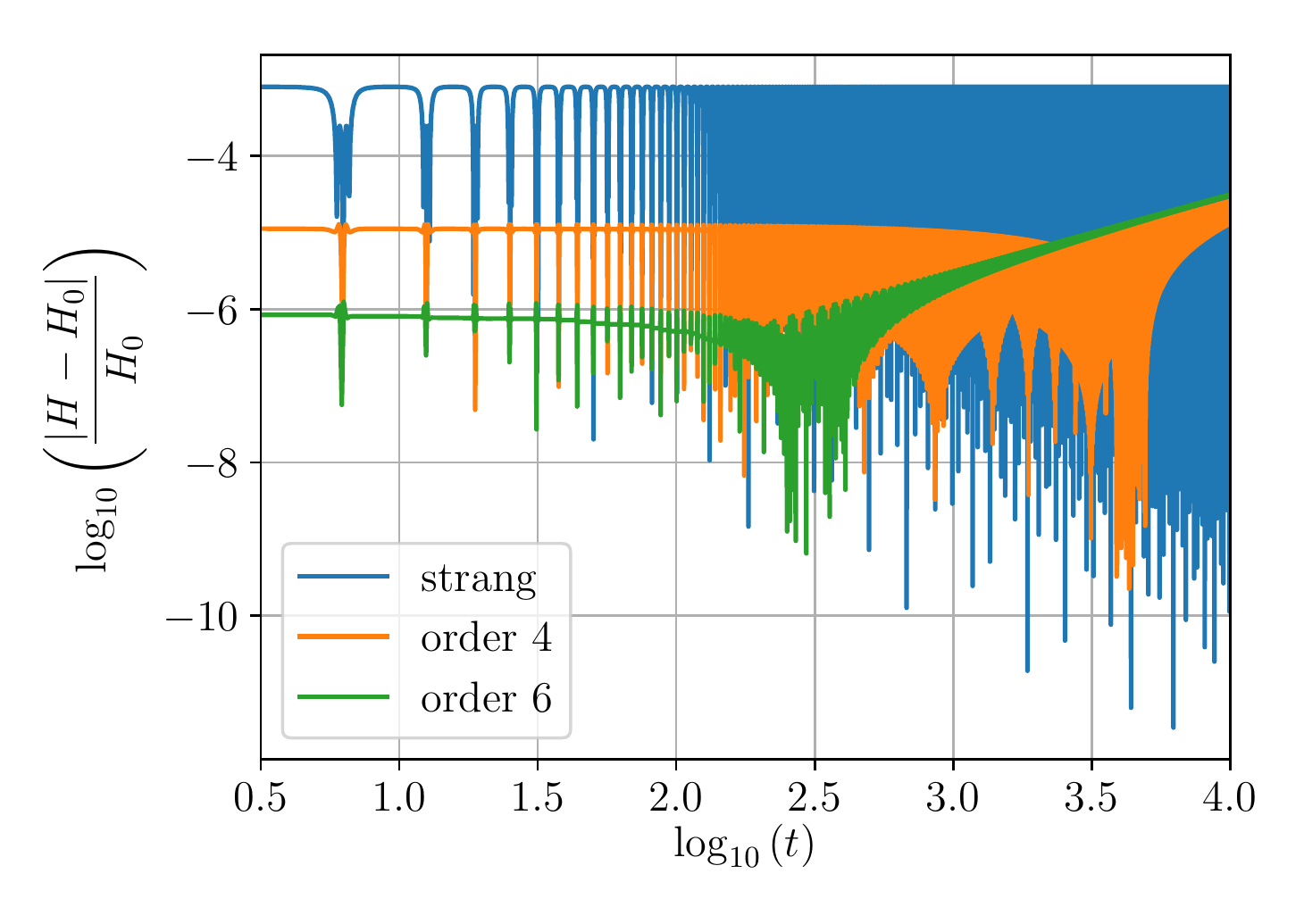} 
\end{center}
\caption{\small Left figure: Relative error in energy vs. the inverse of the step size $\tau$ after approximately $3.183$ periods ($t=20$) for the Kepler problem for the
schemes obtained by the recurrence (\ref{foit}). Right figure:
Evolution of this error along the integration. \label{fig:kepler}}
\end{figure}

\subsection{The semi-linear reaction-diffusion equation of Fisher}
Our third test-problem is  the scalar equation in one-dimension 
\begin{eqnarray} \label{eq:nlrd}
\frac{\partial u(x,t)}{\partial t}  = \Delta u(x,t)  + F(u(x,t)) ,
\end{eqnarray}
with periodic boundary conditions on the interval $[0,1]$. Here $F(u)$ is a \emph{nonlinear} reaction term. For the purpose of testing our methods, we take Fisher's potential \cite{sari15fie}
$$F(u)=u(1-u)$$ 
as considered for example in \cite{blanes13oho}. 

The splitting corresponds here to solving, on the one hand,  the linear equation with the Laplacian (as $f_a$), and on the other hand, the non-linear ordinary differential equation
$$
\frac{\partial u(x,t)}{\partial t} = u(x,t)(1-u(x,t)), 
$$
with initial condition $u(x,0)=u_0(x)$, whose analytical solution is given by the well-defined (for small enough complex time $t$) formula
$$
u(x,t) = u_0(x)+ u_0(x)(1-u_0(x))\frac{(\e^t-1)}{1+u_0(x) (\e^t-1)}.
$$

Here we aim to solve Eq. \eqref{eq:nlrd} with periodic boundary conditions on the interval $[0,1]$, and  initial condition $u_0(x)  = \sin(2\pi x)$. Numerically, the interval  is discretised on a uniform grid, i.e., $x_j = j/N, j  =0,\ldots, N-1,~ N\in \mathbb N$, and $u(x,t)$ is approximated by Fourier pseudo-spectral methods. 
In this way we construct a vector $\mathbf{u}$ with components $(\mathbf{u})_j \approx u(x_{j-1},t)$, $j=1,2,\ldots, N$. If we denote by $\mathbf{u}_{\tau}$
the whole numerical solution computed by a certain integrator with step size $\tau$ from $t=0$ until the final time, and by $\mathbf{u}_{\tau/2}$ the corresponding numerical solution computed by
the same integrator with
step size $\tau/2$, then the quantity $E_\tau := \|\mathbf{u}_\tau- \mathbf{u}_{\tau/2}\|_\infty$ is a good indicator of the convergence order.

Numerical simulations were carried out in quadruple precision (with Intel Fortran) such that roundoff errors are suppressed.  Figure \ref{fig:nlrd} shows the successive errors $E_{\tau}$, at final time $T =10$, of the methods obtained with the sequence (\ref{foit}) with the Strang splitting as the basic method $\mathcal{S}_{\tau}^{[2]}$ (left)
and  the fourth order scheme $\mathcal{S}_{\tau}^{[4]}$  given by (\ref{s4sim}) (right)
with different time steps $\tau_j = 0.1/2^j, ~j = 1,\ldots,7$. One can clearly observe  that the convergence order matches the previous analysis with a slightly better performance for the highest order, analogously to the Kepler problem.  
Figure \ref{fig:nlrd:cost} shows the successive errors versus the number of basic integrators in each case. 

\begin{figure}
\begin{center}
\scalebox{0.9}{\includegraphics{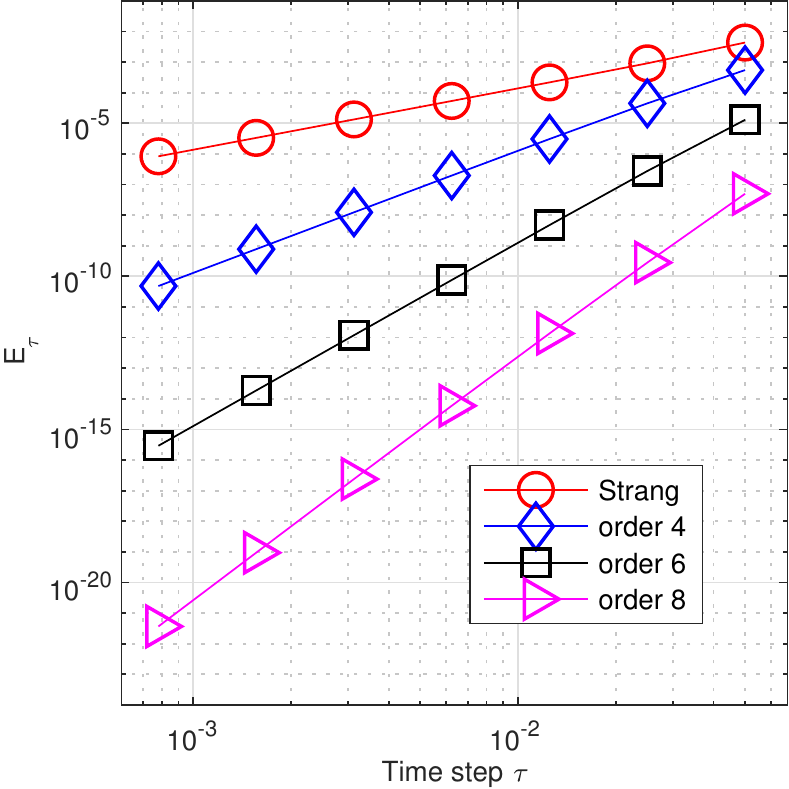}} \scalebox{0.9}{\includegraphics{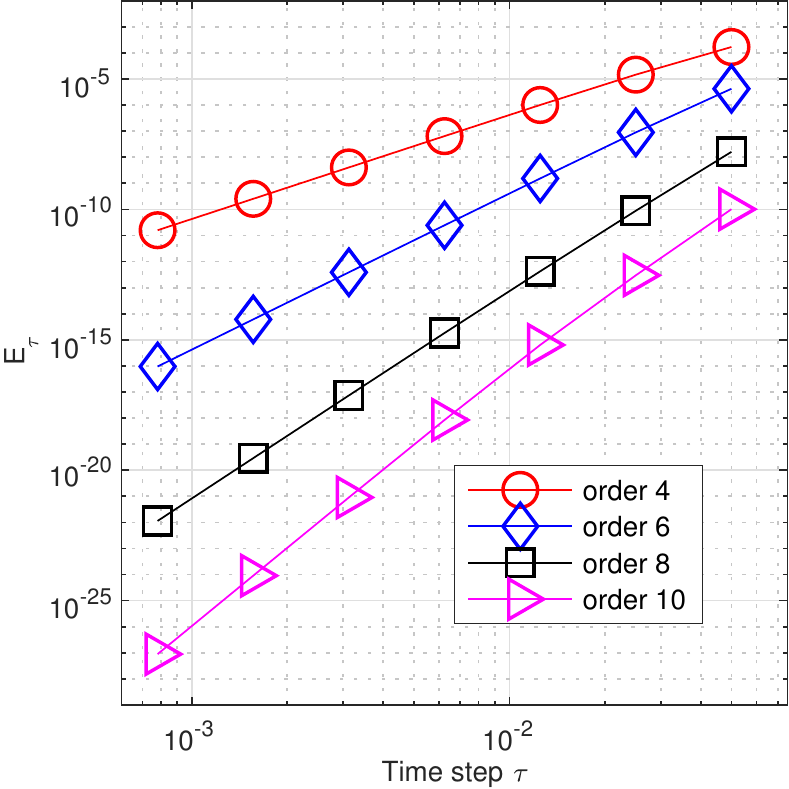}}
\end{center}
\caption{\small Successive errors $E_\tau$ versus time step $\tau$ for Eq. \eqref{eq:nlrd} of the composition methods starting from the Strang scheme (left) and the fourth order scheme $\mathcal{S}_{\tau}^{[4]}$ (right). \label{fig:nlrd}}
\end{figure}  

\begin{figure}
\begin{center}
\scalebox{0.9}{\includegraphics{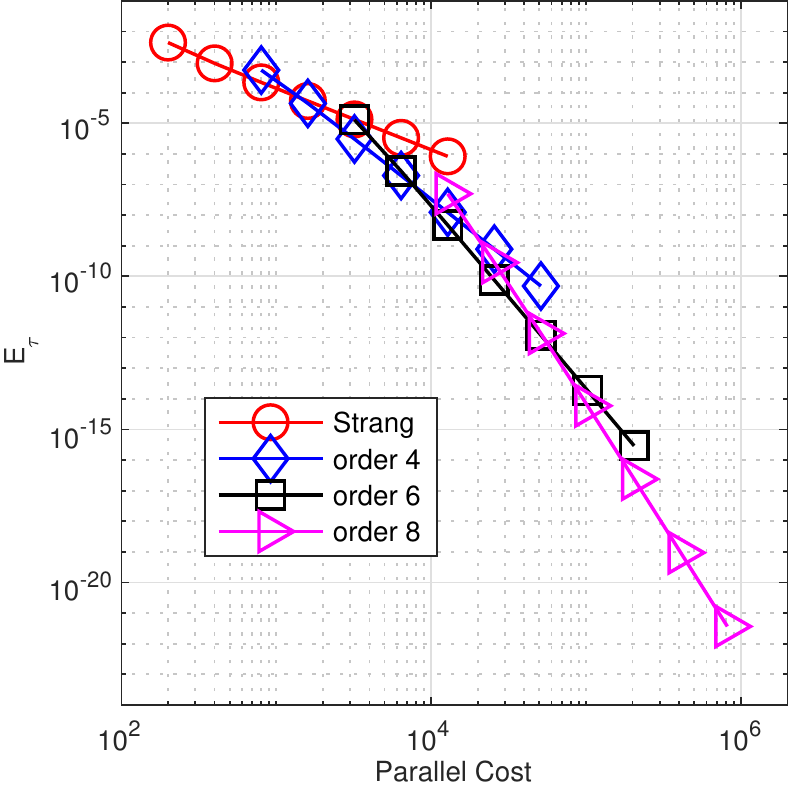}} \scalebox{0.9}{\includegraphics{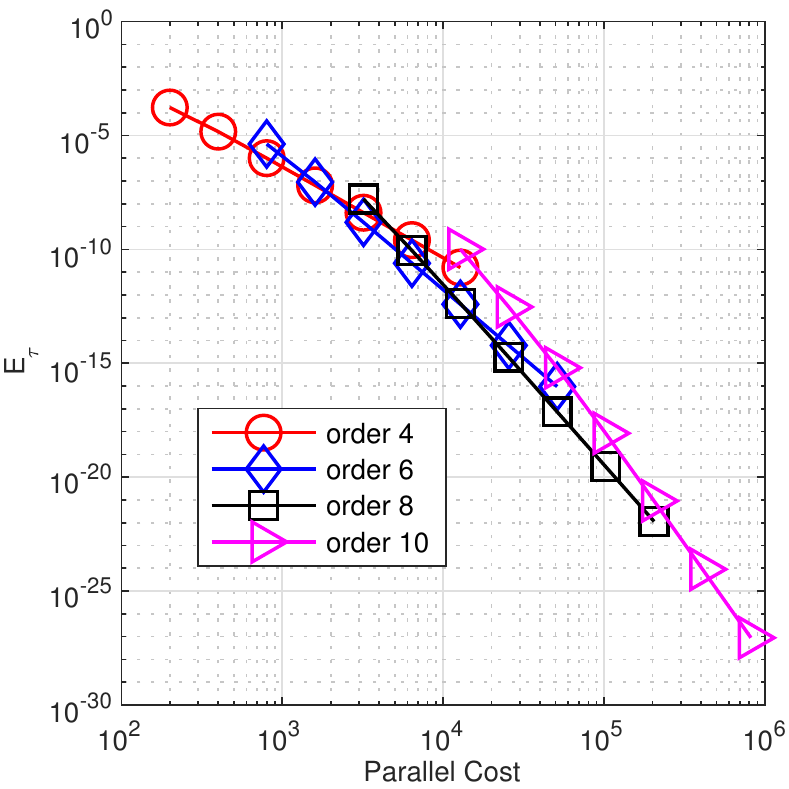}}
\end{center}
\caption{\small Successive errors $E_\tau$ versus number of basic integrators for Eq. \eqref{eq:nlrd} of the composition methods starting from the Strang scheme (left) and the fourth order scheme $\mathcal{S}_{\tau}^{[4]}$ (right). \label{fig:nlrd:cost}}
\end{figure}

%
%
%

\subsection{The semi-linear complex Ginzburg--Landau equation}
Our final test problem is the complex Ginzburg--Landau equation on the domain $(x,t) \in [-100,100] \times [0,100]$,
\begin{eqnarray} \label{eq:CGL}
\frac{\partial u(x,t)}{\partial t}  &=& \alpha \Delta u(x,t) + \varepsilon u(x,t) - \beta|u(x,t)|^2 u(x,t),
\end{eqnarray}
with $\alpha=1+i c_1$, $\beta=1-i c_3$ and initial condition $u(x,0) = u_0(x)$. Here,  $\varepsilon$, $c_1$ and $c_3$ denote \emph{real} coefficients. In physics, the Ginzburg-Landau appears in the  mathematical theory used to model superconductivity. For a broad introduction to the rich dynamics of this equation, we refer to \cite{saarloos95tcg}. Here, we will use the values $c_1=1$, $c_3=-2$ and $\varepsilon=1$, for which plane wave solutions establish themselves quickly after a transient phase (see \cite{winterbottom05opf}). In addition, we set  
$$
u_0(x) = \frac{0.8}{\cosh(x-10)^2} + \frac{0.8}{\cosh(x+10)^2},
$$
so that the solution can be represented in Figure \ref{fig:CGL}. 

\begin{figure}[h!] 
\begin{center}
\scalebox{0.428}{\includegraphics{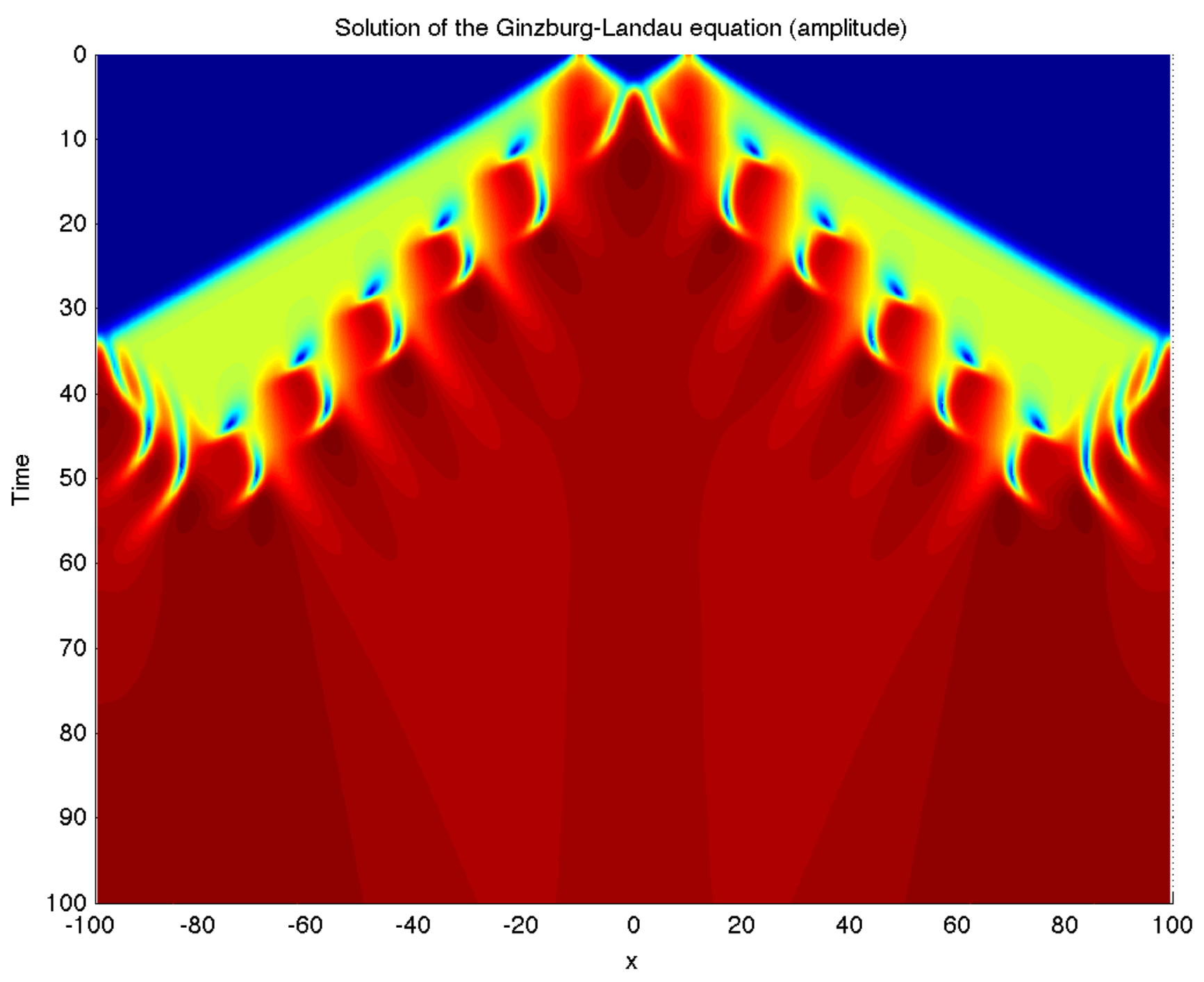}} \scalebox{0.428}{\includegraphics{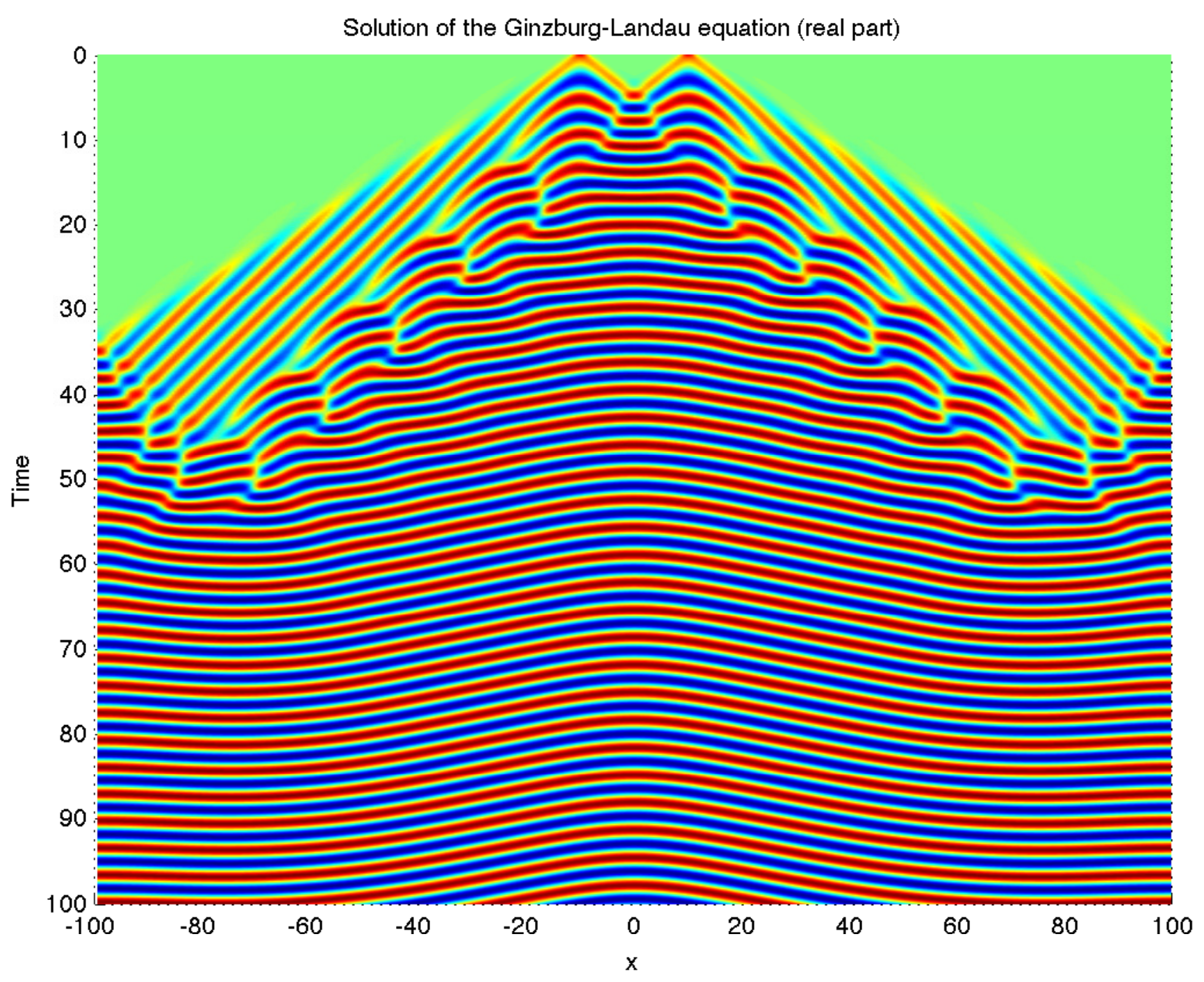}}
\end{center}
\caption{\small Colormaps of the amplitude $|u(x,t)|^2$ (left) and real part $\Re(u(x,t))$ (right) of the solution of (\ref{eq:CGL}). \label{fig:CGL}}
\end{figure}

To apply the composition methods presented in previous sections, it seems natural to split equation (\ref{eq:CGL}) as 
\begin{eqnarray} \label{eq:fp}
\frac{\partial u(x,t)}{\partial t}  = (1+i c_1) \Delta u(x,t) + \varepsilon u(x,t),
\end{eqnarray}
whose solution is $u(x,t)= \e^{\varepsilon t} \e^{t (1+i c_1) \Delta} u_0(x)$ for $t \geq 0$, and 
\begin{eqnarray} \label{eq:sp}
\frac{\partial u(x,t)}{\partial t}  =  - (1-i c_3)  |u(x,t)|^2 u(x,t)
\end{eqnarray}
with solution is for $t \geq 0$
\begin{align*}
u(x,t) =  \e^{- (1-i c_3) \int_0^t M(x,s) ds} u_0(x)  = \e^{- \frac{\beta}{2} \log(1+2 t   M_0(x))} u_0(x). 
\end{align*}
Here we have first solved  the  equation for $M(x,t):=|u(x,t)|^2$, given by 
\begin{eqnarray*}
\frac{\partial M(x,t)}{\partial t}  =  - 2  M^2(x,t),
\end{eqnarray*}
with solution 
\begin{eqnarray*}
M \left(x, t \right) = \frac{M_0(x)}{1+2 M_0(x) t} .
\end{eqnarray*}
Considering $t$ now as a complex variable with positive real part does not raise any difficulty for the first part, since $\e^{\varepsilon t} \e^{t (1+i c_1) \Delta}$ is well-defined. More care has to be taken for the second part, since $u \mapsto |u|^2 u$ is \emph{not} a holomorphic function, and this prevents us from solving (\ref{eq:CGL}) in its current form. As a consequence, we  first rewrite (\ref{eq:CGL}) as a system for $(v(x,t),w(x,t))$ where $v(x,t) =\Re(u(x,t))$ and $w(x,t)=\Im(u(x,t))$: 
\begin{eqnarray} \label{eq:syst}
\left\{
\begin{array}{rcl}
\displaystyle \frac{\partial v(x,t)}{\partial t} &=& \Delta v(x,t) - c_1 \Delta w(x,t) + \varepsilon v(x,t)  -(v^2(x,t)+w^2(x,t)) (v(x,t)+c_3 w(x,t)) \\
\displaystyle \frac{\partial w(x,t)}{\partial t} &=& c_1 \Delta v(x,t) + \Delta w(x,t) + \varepsilon w(x,t)  -(v^2(x,t)+w^2(x,t)) (-c_3 v(x,t) + w(x,t))
\end{array}
\right.
\end{eqnarray}
and now solve it for complex time $t \in \C$ with $\Re(t) \geq 0$. Observing that 
\begin{eqnarray*}
\left(
\begin{array}{cc}
-1 & -c_3 \\
c_3 & -1
\end{array}
\right)
= P D_3 P^{-1} \mbox{ and }
\left(
\begin{array}{cc}
1 & -c_1 \\
c_1 & 1
\end{array}
\right)
=P D_1P^{-1},
\end{eqnarray*}
with 
$$
D_1=
\left(
\begin{array}{cc}
\alpha & 0 \\
0 & \bar{\alpha}
\end{array}
\right), 
D_3=\left(
\begin{array}{cc}
-\beta & 0 \\
0 & -\bar{\beta}
\end{array}
\right), 
P = 
\left(
\begin{array}{cc}
i & 1 \\
1 & i
\end{array}
\right) \mbox{ and }
P^{-1}=
\left(
\begin{array}{cc}
-\frac{i}{2} & \frac12 \\
\frac12 & -\frac{i}{2}
\end{array}
\right),
$$
system (\ref{eq:syst}) can be rewritten as 
\begin{eqnarray} \label{eq:systbis}
\left\{
\begin{array}{rcl}
\displaystyle \frac{\partial \tilde{v}(x,t)}{\partial t} &=& \Big( \alpha \Delta \tilde{v}(x,t) + \varepsilon \tilde{v}(x,t) \Big) -\Big( \beta \tilde{M}(x,t) \tilde{v}(x,t) \Big)\\
\displaystyle \frac{\partial \tilde{w}(x,t)}{\partial t} &=& \Big( \bar{\alpha} \Delta \tilde{w}(x,t) + \varepsilon \tilde{w}(x,t) \Big) -\Big( \bar{\beta} \tilde{M}(x,t) \tilde{w}(x,t)\Big)
\end{array}
\right.
\end{eqnarray}
where $\tilde{M}(x,t) = 4 i \tilde{v}(x,t) \tilde{w}(x,t)$ and where 
\begin{eqnarray*}
\left(
\begin{array}{c}
\tilde{v} \\
\tilde{w}
\end{array}
\right)= 
\frac12
\left(
\begin{array}{cc}
-i & 1 \\
1 & -i
\end{array}
\right)
\left(
\begin{array}{c}
v \\
w
\end{array}
\right).
\end{eqnarray*}
It is not difficult to see that the exact solution of the second part of (\ref{eq:systbis}) is given by 
\begin{eqnarray} \label{eq:vtwt}
\left\{
\begin{array}{rcl}
\tilde{v}(x,t)  &=& \tilde{v}_0(x) \e^{-\frac{\beta}{2} \log(1+2 t \tilde{M}_0(x))} \\
\tilde{w}(x,t) &=& \tilde{w}_0(x) \e^{-\frac{\bar{\beta}}{2} \log(1+2 t \tilde{M}_0(x))}
\end{array}
\right.
\end{eqnarray}
where $\tilde{M}_0(x)$ is now defined as $\tilde{M}_0(x):=4 i  \tilde{v}_0(x) \tilde{w}_0(x)$. Note that here, by convention, the logarithm refers to the principal value of $\log(z)$ for  complex numbers:  if $z=(a+i b)=r \e^{i \theta}$ with $-\pi < \theta \leq \pi$, then 
$$
\log z: = \ln r + i \theta = \ln | z  | + i \arg z =     \ln(|a + ib|) + 2i\arctan \left(\frac{b}{a + \sqrt{a^2+b^2}}\right).
$$
Since $\log(z)$ is \emph{not defined} for $z \in \R^{-}$, this means that the solution $(\tilde{v}(x,t), \tilde{w}(x,t))$ is defined only as long as $1+2 \tilde{M}_0(x) t \notin \R^{-}$. Finally, the solution $(v(x,t),w(x,t))$ is of the form
\begin{eqnarray*}
\left\{
\begin{array}{rcl}
v(x,t)  &=& v_0(x) \frac{(\e^{-\beta L(x,t)} + \e^{-\bar{\beta} L(x,t)})}{2} - w_0(x) \frac{(\e^{-\beta L(x,t)} - \e^{-\bar{\beta} L(x,t)})}{2i}  \\
w(x,t) &=& v_0(x) \frac{(\e^{-\beta L(x,t)} - \e^{-\bar{\beta} L(x,t)})}{2i} + w_0(x) \frac{(\e^{-\beta L(x,t)} + \e^{-\bar{\beta} L(x,t)})}{2}
\end{array}
\right.
\end{eqnarray*}
where $L(x,t):= \log(1+2 t \tilde{M}_0(x))=\log(1+2 t M_0(x))$ with $M_0(x)=v_0^2(x)+w_0^2(x)$.\\ 

Denoting $V=(v_1, \ldots, v_N) \in \R^N$ and $W=(w_1, \ldots, w_N) \in \R^N$, we eventually have to numerically solve the following system:
\begin{eqnarray*} 
\left\{
\begin{array}{rcl}
\dot{V} &=& AV - c_1 A W + \varepsilon V - G (V+c_3 W) \\
\dot{W}  &=& c_1 A V +A W + \varepsilon W  - G (-c_3 V + W)
\end{array}
\right.
\end{eqnarray*}
where $G$ is the diagonal matrix with $G_{i,i}=v_i^2+w_i^2$. \\


Equation \eqref{eq:CGL} is solved with periodic boundary conditions on the interval $[-100,100]$. Now, in the previous example,
the interval  is discretised on a uniform grid, i.e., $x_j = j/N, j  =0,\ldots, N-1,~ N\in \mathbb N$ with $N =512$, and $u(x,t)$ is approximated by Fourier pseudo-spectral methods. The successive errors $E_\tau := \|\mathbf{u}_\tau- \mathbf{u}_{\tau/2}\|_\infty$ are shown also here to confirm the convergence order.  
Figure \ref{fig:cgl} shows the successive errors, at final time $T =10$, of the schemes obtained by applying the sequence (\ref{foit}) from
the basic Strang splitting and the fourth-order scheme (\ref{s4sim}) with $\tau_j = 0.1/2^j, ~j = 1,\ldots,7$. The observed order
of convergence matches the previous analysis with a slightly better performance for the highest order.  Figure \ref{fig:cgl:cost} shows the successive errors versus the number of basic integrators.

\begin{figure}
\begin{center}
\scalebox{0.9}{\includegraphics{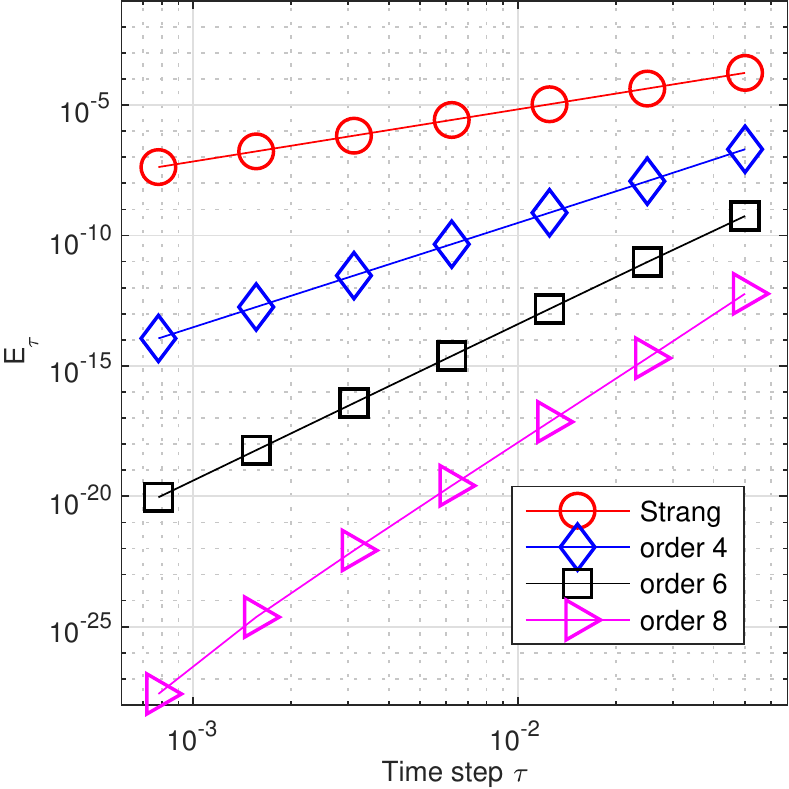}} \scalebox{0.9}{\includegraphics{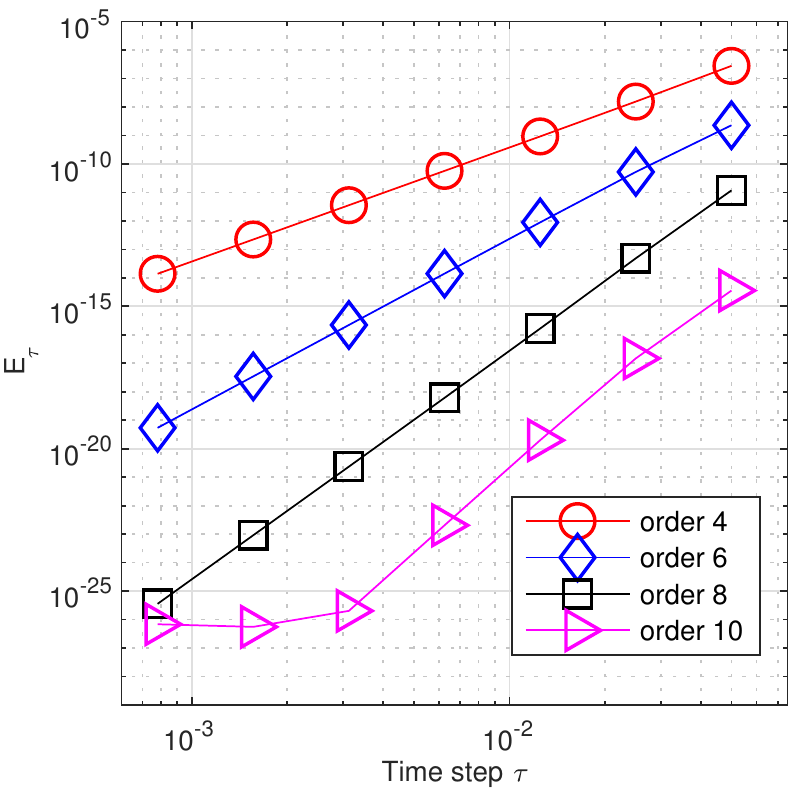}}
\end{center}
\caption{\small Successive errors $E_\tau$ versus time step $\tau$ for Eq. \eqref{eq:CGL} of the composition methods starting from the Strang scheme (left) and the fourth order scheme $\mathcal{S}_{\tau}^{[4]}$ (right). \label{fig:cgl}}
\end{figure}  

\begin{figure}
\begin{center}
\scalebox{0.9}{\includegraphics{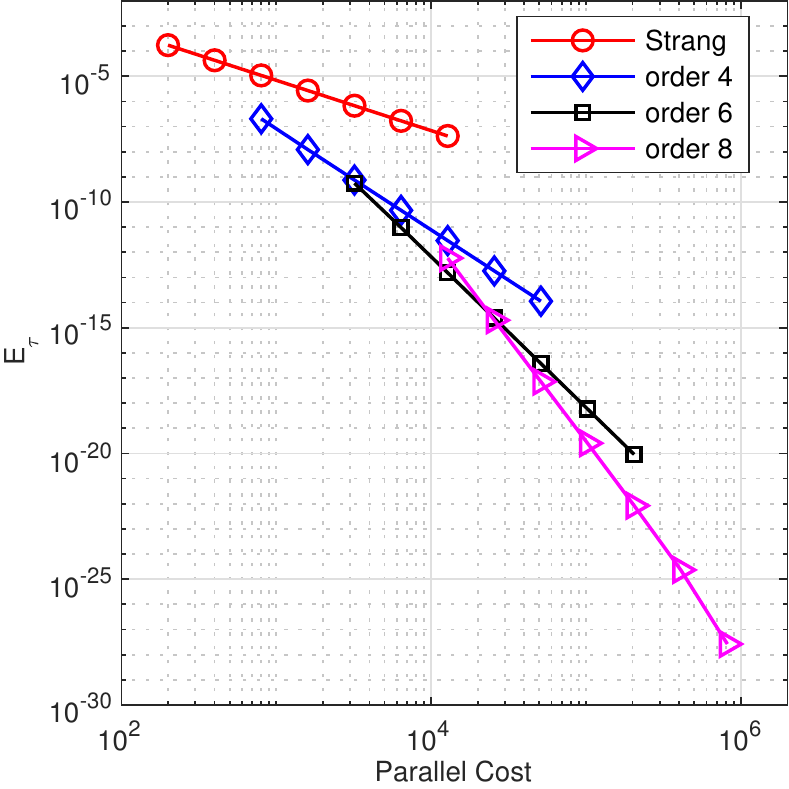}} \scalebox{0.9}{\includegraphics{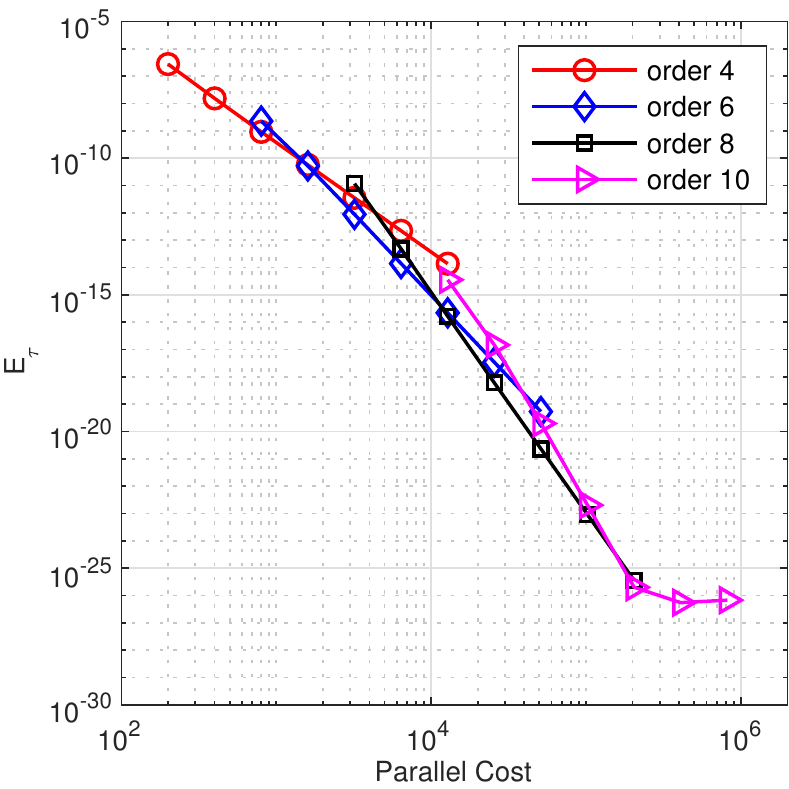}}
\end{center}
\caption{\small Successive errors $E_\tau$ versus number of basic integrators for Eq. \eqref{eq:CGL} of the composition methods starting from the Strang scheme (left) and the fourth order scheme $\mathcal{S}_{\tau}^{[4]}$ (right). \label{fig:cgl:cost}}
\end{figure}

\subsection*{Acknowledgements}
The work of the first three authors has been supported by Ministerio de Econom{\'i}a y Com\-pe\-ti\-ti\-vi\-dad (Spain) through project MTM2016-77660-P (AEI/FEDER, UE).  
PC acknowledges funding by INRIA through its Sabbatical program and thanks the University of the Basque Country for its hospitality. This work was initiated during YZ's visit at the University of Rennes, IRMAR.

\appendix

\section*{Appendix}

In this Appendix we provide an alternative proof of Proposition \ref{prop.2} via Lie formalism. This allows us not only to gain some additional 
insight into the structure of the methods, but also to generalize the result on pseudo-symplecticity to other properties of geometric character, very often
related to Lie groups, the differential equation may possess.

To begin with, if $\varphi_{\tau}$ is the exact flow of the equation (\ref{de.1}), then for each  infinitely differentiable map $g$,
 the function $g(\varphi_{\tau}(x))$ admits an expansion of the form \cite{arnold89mmo,sanz-serna94nhp}
\[
g(\varphi_{\tau}(x)) = \exp(\tau F)[g](x) = g(x) + \sum_{k\geq 1}
\frac{\tau^k}{k!} F^k[g](x), 
\]
where $F$ is the Lie derivative associated with $f$,
\begin{equation}  \label{eq.1.1b}
  F = \sum_{i\ge1} \, f_i(x) \, \frac{\partial }{\partial x_i}.
\end{equation}
Analogously, for a given integrator $\psi_{\tau}$ one can associate a series of linear operators so that
\[
    g(\psi_{\tau}(x)) = \exp(Y(\tau))[g](x), \quad \mbox{ with }  \quad Y(\tau) = \sum_{j \ge 1} \tau^j Y_j
\]
for all functions $g$ \cite{blanes08sac}. The integrator $\psi_{\tau}$ is of order $k$ if
\[ 
   Y_1 = F \qquad \mbox{ and } \qquad Y_j = 0 \;\; \mbox{ for } \;\; 2 \le j \le k.
\]   
For the adjoint integrator $\psi_{\tau}^* = \psi_{\tau}^{-1}$, one clearly has
\[
   g(\psi_{\tau}^*(x)) = \exp \big(-Y(-\tau) \big)[g](x).
\]   
This shows that $\psi_{\tau}$ is symmetric if and only if $Y(\tau) = \tau Y_1 + \tau^3 Y_3 + \cdots$, and in particular, that symmetric methods
are of even order.

An integrator $\mathcal{S}_{\tau}^{[2n]}$ of order $2n  \ge 2$ can be associated with the series
\begin{equation}  \label{s2.1}
  \Phi(\tau) = \exp \left( \tau F + \tau^{2n+1} N_{2n+1} +  \tau^{2n+2} N_{2n+2} +  \cdots \right)
\end{equation}
for certain operators $N_k$. Then, the adjoint method $(\mathcal{S}_{\tau}^{[2n]})^*$  has the associated series
\[
  \Phi^*(\tau) = \exp \left( \tau F + \tau^{2n+1} N_{2n+1}  - \tau^{2n+2} N_{2n+2} +  \cdots \right).
\]
In consequence, $\mathcal{S}_{\tau}^{[2n]}$ is pseudo-symmetric of order $q \ge 2n+1$.  

\

\noindent
(i) Let us analyse first the case $q > 2n+1$. Then, $N_{2n+2} = \cdots = N_{q} = 0$ in (\ref{s2.1}) and
the series of operators associated with the composition 
$ \psi_{\tau}^{[2n+1]} = \psis_{\gamma_1 \tau}^{[2n]} \circ \, \psis_{\gamma_2 \tau}^{[2n]}$ is
\[
   \Psi(\tau) = \Phi(\gamma_2 \tau) \, \Phi(\gamma_1 \tau) \equiv \exp( V(\tau) ),
\]
where $V(\tau)$ can be formally  determined by applying the Baker--Campbell--Hausdorff formula \cite{blanes16aci} as
\[
\begin{aligned}
  & V(\tau)   =  (\gamma_1 + \gamma_2) \tau F + (\gamma_1^{2n+1} + \gamma_2^{2n+1}) \tau^{2n+1} N_{2n+1}  \\
  &  \quad  +  \frac{1}{2} ( \gamma_2 \gamma_1^{2n+1} -  \gamma_1 \gamma_2^{2n+1}) \tau^{2n+2} [F, N_{2n+1}] + 
     (\gamma_1^{2n+3} + \gamma_2^{2n+3}) \tau^{2n+3} N_{2n+3} + \mathcal{O}(\tau^{2n+4}).
\end{aligned}
\]
Here $[\cdot , \cdot]$ denotes the usual Lie bracket.
Clearly, the order of $ \psi_{\tau}^{[2n+1]}$ is $2n+1$ if  
\begin{equation}  \label{gap1}
   \gamma_1 + \gamma_2 = 1,  \qquad \gamma_1^{2n+1} + \gamma_2^{2n+1} = 0,
\end{equation}
so that $\gamma_2 = \bar{\gamma}_1 \equiv \gamma$ is given by eq. (\ref{coefi.com}) (with $k=2n$).  In that case we can write
\[
  V(\tau) = \tau  F + \tau^{2n+2} G(\tau), \qquad \mbox{ with } \qquad G(\tau) = \sum_{i=0}^{\infty} \tau^i G_i,
\]
whereas for the adjoint method one has
\[
   \Psi^{*}(\tau) =  \exp( -V(-\tau) )= \exp \left(  \tau  F + \tau^{2n+2} \widetilde{G}(\tau)  \right),  
   \qquad \mbox{ with } \qquad \widetilde{G}(\tau) = \sum_{i=0}^{\infty} (-1)^{i+1} \tau^i G_i.
\]
In particular, $G_0 =  \frac{1}{2} ( \gamma_2 \gamma_1^{2n+1} -  \gamma_1 \gamma_2^{2n+1})  [F, N_{2n+1}] $, and 
$G_1 = (\gamma_1^{2n+3} + \gamma_2^{2n+3})  N_{2n+3}$.

The series $ \Psi(\tau)$ can also be written as
\[
   \Psi(\tau) =  \exp \left( \frac{\tau}{2} F \right) \exp W(\tau) \exp \left( \frac{\tau}{2} F \right), 
\]
where $W(\tau)$ is determined by applying the symmetric BCH formula \cite{blanes16aci} as
\[
\begin{aligned}
   &  W(\tau) = \tau^{2n+2} G(\tau) + \frac{1}{24} \tau^{2n+4} [F,[F,G(\tau)]] + \mathcal{O}(\tau^{4n+4})    \\
   & \quad = \tau^{2n+2} G_0 + \tau^{2n+3} G_1 + \tau^{2n+4} \big(G_2 + \frac{1}{24} [F,[F,G_0]] \big)+  \mathcal{O}(\tau^{2n+5}). 
\end{aligned}
\]
By the same token,
\[
   \Psi^*(\tau) =  \exp \left( \frac{\tau}{2} F \right) \exp \big(-W(-\tau) \big) \exp \left( \frac{\tau}{2} F \right). 
\]
Consider now the method
\begin{equation}   \label{mad1}
     R_{\tau} = \frac{1}{2} \big( \psi_{\tau}^{[2n+1]} + (\psi_{\tau}^{[2n+1]})^* \big).
\end{equation}
Clearly, its associated series of operators,
\[
  \mathcal{R}(\tau) = \frac{1}{2}  \Psi(\tau) + \frac{1}{2} \Psi^*(\tau),
\]
can be expressed as
\[
    \mathcal{R}(\tau) = \exp \left( \frac{\tau}{2} F \right) \mathcal{Y} \,  \exp \left( \frac{\tau}{2} F \right),
\]
where 
\[
    \mathcal{Y} = \frac{1}{2} \exp \left(  W(\tau) \right)  + \frac{1}{2} \exp \left( -W(-\tau) \right).    
\]
By expanding, we have        
\[
    \mathcal{Y} = I + \frac{1}{2}  \big( W(\tau) - W(-\tau) \big) + \frac{1}{4}  \big( W^2(\tau) + W^2(-\tau) \big) + \cdots,
\]
but
\[
  W(\tau) - W(-\tau) = 2 \tau^{2n+3} \sum_{i=0}^{\infty} \tau^{2i} z_{2i}  \equiv 2 \tau^{2n+3} Z(\tau),
\]
with $z_0 = G_1$, $z_2 = G_3 + \frac{1}{24} [F,[F, G_1]]$, etc. In general, $z_{2i}$ is a linear combination of the operators
$\{F, N_{2n+1}, N_{2n+2}, \ldots\}$ and their nested Lie brackets. In addition, $W^2(\tau) + W^2(-\tau) = \mathcal{O}(\tau^{4n+4})$,
so that we can write
\[
   \mathcal{Y} = I + \tau^{2n+3} Z + \mathcal{O}(\tau^{4n+4}) = \exp  \left( \tau^{2n+3} Z \right) + \mathcal{O}(\tau^{4n+4})
\]   
and 
\[
    \mathcal{R}(\tau) = 
      \exp \left( \frac{\tau}{2} F \right) \, \exp \left( \tau^{2n+3} Z \right) \, \exp \left( \frac{\tau}{2} F \right) +  \mathcal{O}(\tau^{4n+4}),
\]
whence the following statements follow at once:
\begin{itemize}
  \item   Method (\ref{mad1}) is of order $2n+2$, since $ \mathcal{R}(\tau) = \exp( \tau F) +  \mathcal{O}(\tau^{2n+3})$.
  \item Since $Z(\tau)$ only contains even powers of $\tau$ (up to $\tau^{q+1} N_{q+1}$), 
  then $\e^{\frac{\tau}{2} F} \e^{\tau^{2n+3} Z} \e^{\frac{\tau}{2} F}$ is a symmetric composition and
  $R_{\tau}$ is pseudo-symmetric of order $\min(q, 4n+3)$.
  \item Let us suppose that scheme (\ref{mad1}) is applied to a Hamiltonian system and that $\psis_{\tau}^{[2n]}$ is of pseudo-symplecticity order $r$. 
  Since $Z$ is an operator in the free Lie algebra
  generated by $\{ F, N_{2n+1}, N_{2n+2}, \ldots \}$, clearly the composition $\e^{\frac{\tau}{2} F} \e^{\tau^{2n+3} Z} \e^{\frac{\tau}{2} F}$ is symplectic
  (at least up to terms  $\mathcal{O}(\tau^r)$). As a matter
  of fact, this can be extended to any geometric property the differential equation (\ref{de.1}) has: volume-preserving, unitary, etc., as long as the basic scheme
  $\psis_{\tau}^{[2n]}$ preserves this property up to order $r$.
\end{itemize}
Finally, in view of (\ref{eq:conjadjointq})-(\ref{eq:conjadjoint}) and recalling that $q \ge 2n+2$, the same considerations apply if we take the complex conjugate instead of the adjoint, i.e., to the
scheme
\begin{equation}  \label{mcc1}
     \hat R_{\tau} = \Re( \psi_{\tau}^{[2n+1]} ) =  \frac{1}{2} \left( \psi_{\tau}^{[2n+1]} + \overline{\psi}_{\tau}^{[2n+1]} \right). 
\end{equation}

\noindent
(ii) We analyse next the case $q = 2n+1$. Then $N_{2n+2} \ne 0$ in (\ref{s2.1}) and, if
$\gamma_1$ and $\gamma_2$ verify equations (\ref{gap1}), then $V(\tau)$ read
\[
 V(\tau)   =   \tau F +   \tau^{2n+2} V_0 + \mathcal{O}(\tau^{2n+3})
\]
with 
\[ 
   V_0 =  (\gamma_1^{2n+2} + \gamma_2^{2n+2}) N_{2n+2}  
+  \frac{1}{2} ( \gamma_2 \gamma_1^{2n+1} -  \gamma_1 \gamma_2^{2n+1})  [F, N_{2n+1}].
\]
Notice that, whereas $\gamma_1^{2n+2} + \gamma_2^{2n+2}$ is a real
number, $\gamma_2 \gamma_1^{2n+1} -  \gamma_1 \gamma_2^{2n+1}$ has non-vanishing real and imaginary parts. In any event,
the same procedure as in the previous case can be carried out, leading to the conclusion that method (\ref{mad1}) is still of order $2n+2$. 

The situation is different, however, for method (\ref{mcc1}), since relations (\ref{eq:conjadjointq})-(\ref{eq:conjadjoint}) do not provide
further information. We have to analyse instead
\[
   \Re( \Psi(\tau))  =  \exp \left( \frac{\tau}{2} F \right) \Re \big( \exp W(\tau) \big) \exp \left( \frac{\tau}{2} F \right), 
\]
with $ W(\tau) =  \tau^{2n+2} V_0+ \mathcal{O}(\tau^{2n+3})$. Noting that
\[
   \Re \big( \exp W(\tau) \big) = I + \Re \big( W(\tau) \big) + \mathcal{O}(\tau^{4n+4}) = I + \tau^{2n+2} \Re(V_0) + \mathcal{O}(\tau^{2n+3})
\]
then we can write
\[
     \Re( \Psi(\tau)) = \exp \left( \frac{\tau}{2} F \right) \, \exp \Big( \tau^{2n+2}  \Re(V_0) + \mathcal{O}(\tau^{2n+3}) \Big) \, 
         \exp \left( \frac{\tau}{2} F \right) + \mathcal{O}(\tau^{4n+4}).
\]
In consequence, $\hat{R}_{\tau}$ is a method of order $2n+1$, pseudo-symmetric of order $2n+1$ and pseudo-symplectic of order
$\min(r,2n+1)$.

\bibliographystyle{alpha}

\end{document}